\documentclass[11pt]{amsart}
\usepackage{amssymb, latexsym, amsmath, amsfonts, hyperref, enumitem, fancyhdr}

\newtheorem{thm}{Theorem}[section]
\newtheorem{cor}[thm]{Corollary}
\newtheorem{lem}[thm]{Lemma}
\newtheorem{prop}[thm]{Proposition}
\theoremstyle{definition}

\theoremstyle{remark}

\numberwithin{equation}{section}

\hypersetup{
  colorlinks,
  citecolor=black,
  linkcolor=black,
  urlcolor=blue}

\newcommand{\mbb}{\mathbb}
\newcommand{\ra}{\rightarrow}
\newcommand{\z}{\zeta}
\newcommand{\pa}{\partial}
\newcommand{\ov}{\overline}

\newcommand{\ep}{\epsilon}
\newcommand{\no}{\noindent}
\newcommand{\al}{\alpha}
\newcommand{\Om}{\Omega}
\newcommand{\cal}{\mathcal}

\begin{document}
\title{Quadrature domains in $\mathbb C^n$}
\thanks{The first named author was supported by a UGC--CSIR Grant. The second named author was supported by the DST SwarnaJayanti Fellowship 2009--2010 and a UGC--CAS Grant}
%\subjclass{Primary: 32F45  ; Secondary : 32Q45}
\author{Pranav Haridas and Kaushal Verma}

\address{Pranav Haridas: Department of Mathematics, Indian Institute of Science, Bangalore 560 012, India}
\email{pranav10@math.iisc.ernet.in}

\address{Kaushal Verma: Department of Mathematics, Indian Institute of Science, Bangalore 560 012, India}
\email{kverma@math.iisc.ernet.in}

\pagestyle{headings}

\begin{abstract}
We prove two density theorems for quadrature domains in $\mathbb C^n$, $n \ge 2$. It is shown that quadrature domains are dense in the class of all product domains of the form $D \times 
\Omega$, where $D \subset \mathbb C^{n-1}$ is a smoothly bounded domain satisfying Bell's Condition R and $\Omega \subset \mathbb C$ is a smoothly bounded domain and also in the class 
of all smoothly bounded complete Hartogs domains in $\mathbb C^2$.
\end{abstract}

\maketitle

\section{Introduction}

\noindent In this note, by a quadrature domain we will mean a bounded domain $G \subset \mathbb C^n$, $n \ge 1$ with finitely many distinct points $q_1, q_2, \ldots, q_p \in G$, 
positive integers $n_1, n_2, \ldots, n_p$ and complex constants $c_{j \al}$ such that
\begin{equation}
\int_G f(z) = \sum_{j = 1}^p \sum_{\vert \alpha \vert =0}^{n_j - 1} c_{j \al} f^{(\al)}(q_j) \label{quadrature identity}
\end{equation}
for every $f$ in the test class $H^2(G)$, the Hilbert space of square integrable holomorphic functions on $G$. Here, as always, $\al = (\al_1, \al_2, \ldots, \al_n)$ is a multi-index, 
$\vert \al \vert = \al_1 + \al_2 + \ldots + \al_n$ is its length and $f^{(\al)} = \pa^{\vert \al \vert} f/ \pa z_1^{\al_1} \pa z_2^{\al_2} \ldots \pa z_n^{\al_n}$. All volume 
integrals, including the one on the left side of (1.1), are with respect to standard Lebesgue measure on $\mathbb R^{2n}$ and this is not being made explicit for want of brevity.
A relation such as this is called a quadrature identity, the points 
$q_1, q_2, \ldots, q_p$ are the nodes while the integer $N = n_1 + n_2 + \ldots + n_p$ is the order of the quadrature identity. The mean value property of holomorphic functions on a 
disc furnishes a prototype example of such an identity, the node being the centre of the disc while the order is one. Various other test classes can be (and have been) considered as 
well along with an appropriate analog of \eqref{quadrature identity}. However, the theory of such domains is perhaps the richest when we consider planar domains with the 
test class of holomorphic functions on 
it along with a suitable integrability hypothesis -- and this is seen in \cite{AS}, \cite{G} and \cite{S} to cite only a few. We were motivated by a question of Sakai 
(in \cite{EGKP}) that asked for a study of such domains in higher dimensions and for this it is only natural to first ensure that this is a sufficiently rich class. The ball and the 
polydisc in $\mathbb C^n$, $n \ge 2$ are easily verified to be quadrature domains by the mean value property of holomorphic functions. To provide further examples in the plane, let us 
recall 
Gustafsson's theorem from \cite{G} which shows that planar quadrature domains are dense in the class of bounded domains whose boundary consists of finitely many smooth curves and this 
was done by working with its associated Schottky double. Thus, there are many more planar quadrature domains than one might be led to believe. A different proof of this theorem using 
the Bergman kernel was given by Bell in \cite{Be-density} where it was suggested that it might be possible to prove a similar density theorem in higher dimensions.
The purpose of this note is to show that this can indeed be done for a broad class of domains in $\mathbb C^n$, $n \ge 2$.  We present two theorems here which generalize the example 
of the ball and the polydisc in $\mathbb{C}^n$ above.

\medskip

Let us first recall a definition.  For a smoothly bounded domain $G \subset \mbb C^n$, $n \ge 1$, let $P_G$ be the orthogonal projection
from $L^2(G)$ onto $H^2(G)$ and $K_G(z, w)$ the corresponding Bergman kernel. Then $G$ is said to satisfy Condition R if
$P_G(C^{\infty}(\ov G)) \subset C^{\infty}(\ov G)$ or equivalently if for
every integer $s \ge 1$ there exists an integer $m(s) \ge 1$ such that   
\[
P : W^{s + m(s)}(G) \ra W^s(G)
\]
is bounded.  Once again, as always, for every $j \ge 1$, the $W^j(G)$'s are the usual Sobolev spaces on $G$. Examples of domains that satisfy this condition include
smoothly bounded planar domains (see for example \cite{Be-book}), strongly pseudoconvex domains in $\mbb C^n$ (\cite{Ko}), pseudoconvex finite type domains in
$\mbb C^n$ (\cite{Ca}), smoothly bounded complete Hartogs domains in $\mathbb C^2$ (\cite{BoSt}) and complete Reinhardt domains in $\mathbb C^n$ (\cite{BeBo}) among others.

\begin{thm} \label{Main theorem}
Let $D \subset \mbb C^{n-1}$, $n \ge 2$, be a smoothly bounded domain satisfying Condition R and $\Om \subset \mbb C$ a smoothly bounded domain. Then there exist 
quadrature domains arbitrarily close to $D \times \Om$.
\end{thm}

\noindent It follows that there are plenty of quadrature domains in all higher dimensions. These will turn out to be arbitrarily small perturbations of $D \times \Omega$ with $D, \Om$ 
as above. Moreover, the perturbing maps will be biholomorphisms that are close to the identity on $D \times \Omega$ and which extend smoothly to the boundary of $D \times \Om$. 
By taking $D = \mbb B^{n-1}$, the unit ball in $\mbb C^{n-1}$ and $\Om \subset \mbb C$ with prescribed connectivity, it follows that there are quadrature domains with 
arbitrary topological complexity. 

\medskip

To describe this construction, note that for each integer $s \ge 0$, the domain $G$ also admits a Bell operator of order $s$ namely, there is a linear differential operator 
$\Phi^s_G$ of order $\nu(s) = s(s+1)/2$ with coefficients in $C^{\infty}(\ov G)$ such that
\[
\Phi^s_G : W^{s + \nu(s)}(G) \ra W^s_0(G)
\]
is bounded and $P_G \Phi^s_G = P_G$. Here $W^j_0(G)$ is the closure of $C^{\infty}_0(G)$ in $W^j(G)$.

\medskip

For a multi-index $\al = (\al_1, \al_2, \ldots, \al_n)$ with $\vert \al \vert \ge 0$, let
\footnote{Recall that $K_G(z, w)$ is conjugate holomorphic in $w$ and hence this should not cause any confusion with the notation for mixed partial derivatives introduced just after 
(1.1). In any case, what is meant will be clear from the functions that are involved.}
\[ 
K^{(\al)}_G(z, w) = \pa^{\vert \al \vert} K_G(z, w) / \pa {\ov w_1}^{\al_1} \pa {\ov w_2}^{\al_2} \ldots \pa {\ov w_n}^{\al_n}
\]
(recall that $K_G(z, w)$ is conjugate holomorphic in $w$!) and $K^{(0)}_G(z, w) = K_G(z, w)$. The complex linear span of all functions of $z$ of the form $K^{(\al)}_G(z, a)$ where $a$ 
varies over $G$ and $\al$ varies over all possible multi-indices with $\vert \al \vert \ge 0$ is the Bergman span $\mathcal K_G$ associated to $G$. 
Let $A^{\infty}(G) = \cal O(G) \cap C^{\infty}(\ov G)$.

\medskip

Recent work in \cite{CS} shows that both Condition R and the existence of a Bell operator are properties that remain valid under forming products. Thus, if $D, \Om$ are as in 
Theorem 1.1, then $D \times \Om$ also satisfies Condition R and admits a Bell operator. A direct consequence of this is that $D \times \Om$ satisfies Bell's density lemma 
(\cite{Be-densitylemma}), i.e., the Bergman span $\mathcal K_{D \times \Om}$ is dense in $A^{\infty}(D \times \Om)$. In other words, for a given $h \in A^{\infty}(D \times \Om)$, an 
$\epsilon  > 0$ 
and an integer $M \ge 0$, there exists $k \in \mathcal K_{D \times \Om}$ such that $h - k$ and all its derivatives up to order $M$ are uniformly bounded by $\ep$ on the closure of $D 
\times \Om$.

\medskip

The relevance of the Bergman span becomes evident if the quadrature identity \eqref{quadrature identity} is rewritten using the standard $L^2$ inner product on $G$ which will be 
denoted by angle brackets as usual. Thus we have
\begin{align*}
\langle f, 1\rangle_{G} &= \int_G f(z) \\
       &= \sum_{j = 1}^p \sum_{\vert \alpha \vert =0}^{n_j - 1} c_{j \al} f^{(\al)}(q_j) = \big \langle f, \sum_{j = 1}^p \sum_{\vert \alpha \vert =0}^{n_j - 1} {\ov c}_{j \al} 
K^{(\al)}_G(.,q_j) \big \rangle_{G}
\end{align*}
for all $f \in H^2(G)$ which means that
\[
1 = \sum_{j = 1}^p \sum_{\vert \alpha \vert =0}^{n_j - 1} {\ov c}_{j \al} K^{(\al)}_G(z, q_j)
\]
for all $z \in G$, i.e., the constant function $h(z) \equiv 1$ belongs to the Bergman span associated to $G$. The question of constructing quadrature domains is therefore equivalent to 
finding those domains for which the Bergman span contains the constants. The reader is referred to \cite{Sh} for further discussions along this line. The proof of the 
theorem consists of several steps, the 
first of which is a criterion which ensures that the range of a proper mapping is a quadrature domain. This was already noted in \cite{Be-density}.

\begin{prop} \label{quadrature domains and proper mappings}
Let $\Om_1, \Om_2 \subset \mbb C^n$ be bounded domains and $f : \Om_1 \ra \Om_2$ a proper holomorphic mapping. Then $\Om_2$ is a quadrature domain if and only if the complex Jacobian $u 
= \det[\pa f_i / \pa z_j] \in \mathcal K_{\Om_1}$.
\end{prop}

Note that there are no hypotheses on the boundaries of $\Om_1$ and $\Om_2$. To use this criterion, choose $u \in \mathcal K_{\Om_1}$. Then we must look for a holomorphic mapping 
$f = (f_1, f_2, \ldots, f_n) : \Om_1 \ra \mbb C^n$ such that
\begin{equation}
u = \det[\pa f_i / \pa z_j].  \label{Jacobian}
\end{equation}
Then $\Om_2 = f(\Om_1)$ will be a quadrature domain provided it is possible to ensure that $f$ is proper. To do this, take $\Om_1 = D \times \Om \subset \mbb C^{n-1} \times 
\mbb C$ as 
in Theorem 1.1 and write $z = (z', z_n)$ where $z'= (z_1, z_2, \ldots, z_{n-1})$ and $z_n$ are coordinates on $D$ and $\Om$ respectively. Pick $u \in \mathcal K_{D \times \Om}$ and 
define $f : D \times \Om \ra \mbb C^n$ as
\begin{equation}
f(z) = (z_1, z_2, \ldots, z_{n-1}, g(z)). \label{Graph of f a proper mapping}
\end{equation}
Then \eqref{Jacobian} is equivalent to finding a holomorphic function $g$ such that
\begin{equation}
\frac{\pa g}{\pa z_n} = u. \label{Partial Differential Equation}
\end{equation}
Two cases arise. First, when $\Om$ is simply connected, $g$ is obtained simply by integrating $u$ along a path that lies in a copy of $\Om$ over a given point in $D$ and then repeating 
this over every point in $D$. The choice of path is irrelevant since $\Om$ is simply connected. Moreover, if $u \in \mathcal K_{D \times \Om}$ is close to the constant function $h(z) = 
1$ in $A^{\infty}(D \times \Om)$ then this representation of $g$ shows that $f$ is close to the identity map (and hence proper!) in $A^{\infty}(D \times \Om)$. Thus $f(D \times \Om)$ is 
a quadrature domain that is arbitrarily close to $D \times \Om$. The other case to consider is when $\Om$ is multiply connected. Note that \eqref{Partial Differential Equation} always admits a local holomorphic 
solution near each point in $D \times \Om$. However, global solutions that are obtained by analytic continuation of such a local solution to all of $D \times \Om$ may not be single 
valued in general. This is because $u$ may have non-zero periods around the inner boundary components of $\Om$. It is possible to rectify this problem by subtracting a suitable function 
from $u$ in such a way that the resulting function still belongs to the Bergman span of $D \times \Om$ and this is inspired by a similar construction in \cite{Be-density}.

\begin{prop} \label{Solution to PDE}
For a given $u \in \mathcal K_{D \times \Om}$ , there exists $v \in \mathcal K_{D \times 
\Om}$ such that
\[
\frac{\pa g}{\pa z_n} = v
\]
admits a single valued holomorphic solution $g$. Consequently, with this choice of $g$, \eqref{Graph of f a proper mapping} defines a single valued holomorphic mapping $f : D \times \Om \ra \mbb C^n$. In particular, if $u\in \mathcal K_{D \times \Om}$ is close to the constant function $h(z) \equiv 1$ in $A^{\infty}(D\times\Omega)$, the image $f(D\times \Omega)$ is a quadrature domain that is close to the  $D\times\Omega$.
\end{prop}

Thus in both cases there are plenty of quadrature domains as close to $D \times \Om$ as desired. The non-explicit examples of 
quadrature domains obtained this way will not have a product structure in general. They do however have non-smooth boundaries. 

\begin{thm} \label{Hartogs domains}
Let $ D \in \mathbb{C}^2$ be a smoothly bounded complete Hartogs domains. Then there exist quadrature domains arbitrarily close to $D$. The same is true for the complex ellipsoid
\[
\mathcal E = \big\{ z \in \mathbb C^n : \vert z_1 \vert^{2m_1} + \vert z_2 \vert^{2m_2} + \ldots + \vert z_n \vert^{2m_n} < 1 \big\} 
\]
where $m_1, m_2, \ldots, m_n$ are positive integers.
\end{thm}

%%%%%%%%%%%%%%%%%%%%%%%%%%%%%%%%%%%%%%%%%%%%%%%%%%%%%%%%%%%%%%%%%%%%%%%%%%%%%%%%%%%%%%%%%%%%%%%%%%%%%%%%%%%%%%%%%%%%

\section{Proof of Proposition \ref{quadrature domains and proper mappings}}

\noindent The ideas are similar to those used in the proof of the transformation formula for the Bergman projection under proper holomorphic mappings. Here are the details. 
Let $\Omega_1$ and $\Omega_2$ be bounded domains in $\mathbb{C}^n$, $f : \Omega_1 \rightarrow \Omega_2$ a proper holomorphic mapping and $u = \det[\pa f_i / \pa z_j]$ its complex Jacobian.  
We shall denote by $E$ the set $\{z \in \Omega_1 : u(z) = 0 \}$ which is a proper subvariety of $\Omega_1$. 
There exists a positive integer $m$ such that $f$ is a $m$ to one covering map of $\Omega_1 - E$ onto $\Omega_2 - f(E)$. Let $g_1, g_2,\ldots, g_m$ be the local inverses of $f$  
and $u_1, u_2,\ldots, u_m$ the respective Jacobians all of which are locally well defined on $\Omega_2-f(E)$.

\medskip

Let $\psi_j \in L^{2}(\Omega_j)$ for $j = 1,2$. The first thing to do is to check that $u(\psi_2 \circ f) \in L^{2}(\Omega_1)$ and that 
$\sum_{k=1}^{m} u_k(\psi_1\circ g_k) \in L^{2}(\Omega_2)$. 
Since $\det J_{\mathbb{R}}f = |\det[\partial f_i/\partial z_j]|^2 = |u|^2$, we have the following change of variables formula 
\begin{equation}
m\int\limits_{\Omega_2}|\psi_2|^{2} = m\int\limits_{\Omega_2-f(E)}|\psi_2|^{2} = \int\limits_{\Omega_1-E}|u|^2|\psi_2 \circ f|^{2} =  \int\limits_{\Omega_1}|u|^2|\psi_2 \circ f|^{2} \label{Change of Variables}
\end{equation}
where the first and last equalities hold since $E$ and $f(E)$ have measure zero and the second equality follows because $f$ is a $m$ to one covering map of 
$\Omega_1 - E$ onto $\Omega_2 - f(E)$. This shows that the $L^{2}$-norm of $u(\psi_2\circ f)$ is equal to $\sqrt{m}$ times the $L^2$-norm of $\psi_2$. For the 
second claim, consider the symmetric function $\sum_{k=1}^{m} u_k(\psi_1\circ g_k)$ which is well defined on $\Omega_2-f(E)$. 
Since $f(E)$ is a measure zero set, it is defined almost everywhere on $\Omega_2$. Then, 
\begin{align*}
\int\limits_{\Omega_2}\big\vert \sum_{k=1}^{m} u_k(\psi_1\circ g_k) \big\vert^{2} &= \int\limits_{\Omega_2-f(E)}\big \vert\sum_{k=1}^{m} u_k(\psi_1\circ g_k)\big \vert^{2} \\
&\leq \int\limits_{\Omega_2-f(E)}m\sum_{k=1}^{m} \lvert u_k(\psi_1\circ g_k)  \rvert^{2} \\
&= m \int\limits_{\Omega_1-E}|\psi_1|^{2}
\end{align*}
where the second step follows from the Cauchy-Schwarz inequality. Since $\psi_1 \in L^2(\Omega_1)$ and $m$ is finite, we conclude that  
$\sum_{k=1}^{m} u_k(\psi_1 \circ g_k) \in L^{2}(\Omega_2)$. 

\medskip

Now let us assume that the complex Jacobian $u = \det[\pa f_i / \pa z_j] \in \mathcal K_{\Om_1}$. For $\psi_2 \in H^2(\Omega_2)$, a similar calculation as 
in \eqref{Change of Variables} shows that
\begin{equation}
m\int\limits_{\Omega_2}\psi_2 = m\int\limits_{\Omega_2-f(E)}\psi_2 = \int\limits_{\Omega_1-E}|u|^2(\psi\circ f) =  \int\limits_{\Omega_1}|u|^2(\psi\circ f). \label{Change of Variables 2}
\end{equation}
Continuing with an arbitrary $\psi_2 \in H^2(\Omega_2)$, since $u(\psi_2 \circ f) \in L^{2}(\Omega_1)$ we can re-write \eqref{Change of Variables 2} as 
\begin{align*}
 m\int\limits_{\Omega_2}\psi_2 = \big\langle u(\psi_2\circ f),u \big\rangle_{\Omega_1}
\end{align*} 	
We have assumed that $u \in \mathcal{K}_1$. Hence, we can write
\begin{align*}
u(\zeta) = \sum_{j,\alpha} t_{j\alpha}K_{\Omega_1}^{(\alpha)}(\zeta, b_j) \hspace{10mm} 
\end{align*}
where $t_{j\alpha} \in \mathbb{C}$. Thus \eqref{Change of Variables 2} can be rewritten as
\begin{align*}
\int\limits_{\Omega_1}|u|^2(\psi\circ f) &= \big\langle u(\psi_2\circ f),u \big\rangle_{\Omega_1} \\
&= \big\langle u(\psi_2\circ f),\sum_{j,\alpha} t_{j\alpha}K_{\Omega_1}^{(\alpha)}(\cdot, b_j) \big\rangle_{\Omega_1} \\
&= \sum_{j,\alpha} \overline{t_{j\alpha}}\big\langle u(\psi_2\circ f),K_{\Omega_1}^{(\alpha)}(\cdot, b_j) \big\rangle_{\Omega_1} \\
&= \sum_{j,\alpha} \overline{t_{j\alpha}}  (u(\psi_2\circ f))^{(\alpha)}(b_j) \\
&= \sum_{j,\alpha} c_{j\alpha}\psi_2^{(\alpha)}(a_{j})
\end{align*}
where $a_j = f(b_j)$. Observe that the constants $c_{j\alpha}$ depend only on $t_{j\beta}$, $u^{(\beta)}(b_j)$ and $f^{(\beta)}(b_j)$ 
for $\vert\beta \vert \leq \vert \alpha \vert$.  Since this is true for every $\psi_2 \in H^2(\Omega_2)$, it follows that $\Omega_2$ is a quadrature domain with nodes 
at $a_1, a_2,\ldots, a_l$.

\medskip

Conversely, suppose that $\Omega_2$ is a quadrature domain. For every $\psi_1 \in H^2(\Omega_1)$, 
\[
\sum_{k=1}^{m} u_k(\psi_1 \circ g_k) \in \mathcal{O}(\Omega_2-f(E)) \cap L^{2}(\Omega_2-f(E)).
\]
Since $\sum_{k=1}^{m} u_k(\psi_1 \circ g_k) \in L^{2}(\Omega_2)$, it has a holomorphic extension to $\Omega_2$. Therefore
\begin{align*}
\langle \psi_1, u \rangle_{\Omega_1} &= \int\limits_{\Omega_1} \psi_1\overline{u} = \int\limits_{\Omega_1-E}  \psi_1\overline{u} \\
&= \int\limits_{\Omega_2-f(E)} \sum_{k=1}^{m}(\psi_1\circ g_{k})(\overline{u\circ g_{k}})|u_k|^2 \\
&= \int\limits_{\Omega_2-f(E)} \sum_{k=1}^{m}u_k(\psi_1\circ g_{k})	 \\
&= \big\langle \sum_{k=1}^{m} u_k(\psi_1 \circ g_k), 1 \big\rangle_{\Omega_2} 
\end{align*}
Using the fact that $\Omega_2$ is a quadrature domain, we get
\begin{align*}
\langle \psi_1, u \rangle_{\Omega_1} &= \sum_{j,\alpha} c_{j\alpha}\bigg(\sum_{k=1}^{m}u_k(\psi_1 \circ g_k)\bigg)^{(\alpha)}(a_{j}) \\
\end{align*}
which implies that
\begin{align*}
\langle \psi_1, u \rangle_{\Omega_1} &= \sum_{j,\alpha} t_{j\alpha}\psi_1^{(\alpha)} (b_{j})\\
&= \big\langle \psi_1,\sum_{j,\alpha} \overline{t_{j\alpha}}K_{\Omega_1}^{(\alpha)}(\cdot, b_j) \big\rangle_{\Omega_1}
\end{align*}
The Riesz representation theorem shows that $u = \sum_{j,\alpha} t_{j\alpha}K_{\Omega_1}^{(\alpha)}(\cdot, b_j)$ which in turn implies that $u$ is an element of $\mathcal{K}_{\Omega_1}$.

%%%%%%%%%%%%%%%%%%%%%%%%%%%%%%%%%%%%%%%%%%%%%%%%%%%%%%%%%%%%%%%%%%%%%%%%%%%%

\section{The Density Lemma on Product Domains}

\no We begin by quickly reviewing a few basic facts about function spaces on product domains that will be useful in understanding the Bergman projection and the Bell operator on them. 
Since these have been studied in detail in \cite{CS} we do so solely to make this exposition complete. We will restrict ourselves to product domains each factor of which is 
smoothly bounded though much of what follows holds in far greater generality. For $1 \le j \le k$, let $\Om_j \subset \mbb C^{n_j}$, $n_j \ge 1$, be a smoothly bounded domain and 
let $\Om = \Om_1 \times \Om_2 \times \ldots \times \Om_k \subset \mbb C^n$ be the product domain where $n = n_1 + n_2 + \ldots + n_k$. Then $\Om$ has Lipschitz boundary and a basic 
fact about such domains is the existence of a continuous linear extension operator from $W^s(\Om)$ to $W^s(\mbb C^n)$. Combining this with the Sobolev embedding of $W^s(\mbb C^n)$ 
for large $s$ into a space of smooth functions shows that
\[
C^{\infty}(\ov \Om) = \bigcap_{s = 0}^{\infty} W^s(\Om)
\]
and further, the usual Fr\`{e}chet topology on $C^{\infty}(\ov \Om)$ given by the $C^s$-norms coincides with the Fr\`{e}chet topology given by the $W^s$-norms. If $f_j$ is a complex 
valued function on $\Om_j$, let
\[
f_1 \otimes f_2 \otimes \cdots \otimes f_k
\]
be their tensor product, which by definition is the function
\[
f(z_1, z_2, \ldots, z_k) = f_1(z_1) \cdot f_2(z_2) \cdot \cdots \cdot f_k(z_k)
\]
defined on $\Om$, where $z_j \in \Om_j$. For each $j$, let $X_j$ be a complex vector space of functions on $\Om_j$. Their algebraic tensor product
\[
X_1 \otimes X_2 \otimes \cdots \otimes X_k = \bigotimes_{j =1}^k X_j
\]
is a complex vector space of functions on $\Om$, each element of which is written as a finite linear combination of tensor products of the form $f_1 \otimes f_2 \otimes \cdots \otimes 
f_k$ where $f_j \in X_j$. It is known that (see \cite{H} for example) the algebraic tensor product $\bigotimes_{j=1}^k C^{\infty}_0(\Om_j)$ is dense in the $C^{\infty}$-topology in 
$C^{\infty}_0(\Om)$. Further, the algebraic tensor product $\bigotimes_{j=1}^k C^{\infty}(\ov \Om_j)$ is dense in $C^{\infty}(\ov \Om)$ in the Fr\`{e}chet topology.

\medskip

If $H_j$ is a Hilbert space of functions on $\Om_j$, the algebraic tensor product $H = \bigotimes_{j=1}^k H_j$ comes with a natural inner product that is given on functions which are tensor products as
\[
\langle f_1 \otimes f_2 \otimes \cdots \otimes f_k, g_1 \otimes g_2 \otimes \cdots \otimes g_k \rangle_H = \prod_{j=1}^k \langle f_j, g_j \rangle_{H_j}.
\]
We define the Hilbert Tensor Product $\widehat{\bigotimes}_{j=1}^k H_j$ to be the completion of $H$ under the above inner product. As an example, 
if $\mathcal{F}_{j}(\Om_{j})$ is a function space on $\Om_{j}$ and $\mathfrak{F}_{j}(\Om_{j})$ the Hilbert space of differential forms on $\Omega_{j}$ with coefficients in $\mathcal{F}_{j}(\Om_{j})$, then it can be shown that $$\widehat{\bigotimes}_{j = 1}^{k}\mathfrak{F}_{j}(\Om_{j})\approx \mathfrak{F}(\Omega_{1}\times\ldots\times\Omega_{k})$$ where $\mathfrak{F}(\Omega_{1}\times\ldots\times\Omega_{k})$ is the Hilbert space of differential 
forms on $\Omega_{1}\times\ldots\times\Omega_{k}$ with coefficients in 
$\widehat{\bigotimes}_{j =1}^{k}\mathcal{F}_{j}(\Omega_{j})$ and the homomorphism is a Hilbert space isometry. 
If $T_{j} : H_{j} \rightarrow H_{j}^{'}$ are bounded 
linear maps between Hilbert spaces, then the map 
$$T = T_{1}\otimes \ldots \otimes T_{k} : \bigotimes_{j=1}^k H_j \rightarrow \bigotimes_{j=1}^k H^{'}_j$$ 
given by $T(f_1\otimes f_2 \otimes\ldots\otimes f_{k}) = T_{1}(f_{1})\otimes\ldots\otimes T_{k}(f_{k})$ on decomposable tensors, extends naturally to a bounded linear map 
$$T_{1}\widehat{\otimes} \ldots \widehat{\otimes} T_{k} :\widehat{\bigotimes}_{j=1}^k H_j \rightarrow \widehat{\bigotimes}_{j=1}^k H^{'}_j$$
We shall represent a multi-index $\alpha$ as $\alpha[1] +\ldots+\alpha[k]$ where 
\[
\alpha[j] = (\underbrace{0,\ldots,0}_{n_1+\ldots+n_{j-1}}, \alpha_{j1},\ldots, \alpha_{jn_j},\underbrace{0,\ldots,0}_{n_{j+1}+\ldots+n_k})
\]
denotes the $n_j$ components of $\alpha$ corresponding to $\Omega_j$. Recall the Partial Sobolev norms from \cite{CS}:
\begin{equation}
\Vert f \Vert_{\widetilde{W}^s(\Omega)}^2 =  \sum\limits_{\substack{|\alpha[j]| \leq s  \\  1 \leq j \leq k}} \Vert D^{\alpha}f \Vert_{L^2(\Omega)}^2. \label{Partial Sobelov Norm}
\end{equation}
The Partial Sobolev space $\widetilde{W}^s(\Omega)$ is defined as the completion of $C^{\infty}(\overline{\Omega})$ under the norm $\eqref{Partial Sobelov Norm}$ given above. 
Alternately, we can define the Partial Sobolev space as $\widetilde{W}^s(\Omega) = \widehat{\bigotimes}_{j = 1}^{k}W^{s}(\Omega_j)$. It was shown in \cite{CS} that the usual Sobolev 
spaces and the Partial Sobolev spaces are essentially the same for Lipschitz domains. 

\begin{lem}\label{Partial Sobelov space containment}
If each $\Omega_j$ has Lipschitz boundary, then we have
\[W^{ks}(\Omega) \subsetneq \widetilde{W}^{s}(\Omega) \subsetneq W^{s}(\Omega) \text{, and} \]

\[W^{ks}_0(\Omega) \subsetneq \widetilde{W}^{s}_0(\Omega) \subsetneq W^{s}_0(\Omega).\]
where the inclusions are continuous.
\end{lem}

As a corollary, it is possible to identify the Bergman space and the associated Bergman projection on product domains.	

\begin{lem}\label{Condition R product domains}
Let $\Omega_{j} \subset \mathbb{C}^{n_{j}}$ be bounded pseudoconvex domains. Then,
\begin{enumerate}
\item[(i)] the Bergman space $H^{2}(\Omega) = \widehat{\bigotimes}H_{j}^{2}(\Omega_{j})$.
\item[(ii)]  \label{Condition R product domains:two} The Bergman projection $P_{\Omega} = P_{1}\widehat{\otimes}\ldots\widehat{\otimes} P_{k}$ where $P_j$ is the Bergman projection 
associated to $\Omega_j$.
\item[(iii)]  \label{Condition R product domains:three} If each $\Omega_{j}$ is Lipschitz and satisfies Condition R, then so does $\Omega$. 
\end{enumerate}
\end{lem}
Let us quickly indicate a proof of (iii). Since each $\Omega_j$ satisfies Condition R, it follows that for each $s\geq 0$, there is a non-negative integer $m_j(s)$,  $ 1\leq j \leq k$ 
such that 
\[ P_j : W^{s +m_j(s)} \rightarrow W^{s}(\Omega_j)\]
is continuous. If $m(s) = \max(m_1(s), m_2(s), \ldots, m_{k}(s))$, then each $P_j$ is continuous from $W^{s +m(s)}$ to $W^{s}(\Omega_j)$. It follows from \ref{Condition R product domains:two}  that 
\[P_{\Omega} : \widetilde{W}^{s +m(s)}(\Omega) \rightarrow \widetilde{W}^{s}(\Omega)\]
is continuous and by Lemma \ref{Partial Sobelov space containment}, we conclude that 
\begin{equation}
P_{\Omega} : W^{ks + km(s)}(\Omega) \rightarrow W^{s}(\Omega) \label{Bergman projection and Sobelov spaces}
\end{equation}
is continuous. This shows that the product domain $\Omega$ also satisfies Condition R if each factor does.	

\begin{thm}
Let $\Omega = \Omega_{1} \times \ldots \times \Omega_k$ be a product of smoothly bounded domains each of which satisfies Condition R. Then the linear span of 
$\{K_{\Omega}(z, a) : a \in \Omega\}$ is dense in $A^{\infty}(\Omega)$ in the $C^{\infty}$-topology.
\end{thm}

\begin{proof}
For $s\geq 0$ and $1\leq j \leq k$, let 
\[\Phi^{s}_{j} : W^{s + \nu(s)}(\Omega_j) \rightarrow W^s_0{(\Omega_j)} \]
be the Bell operator of order $s$ for $\Omega_j$ that satisfies $P_j \Phi_j^{s} = P_j$ where $P_{j}$ is the Bergman projection associated to $\Omega_j$. The Bell operator for $\Omega$ of order $s$ may be obtained in the following way. Define
\[ \Phi^{s}_{\Omega} = \Phi_1^s\widehat{\otimes}\ldots\widehat{\otimes}\Phi_m^s \]
which is apriori continuous from $\widetilde{W}^{s +\nu(s)}(\Omega)$ to $\widetilde{W}^{s}_{0}(\Omega)$. By Lemma \ref{Partial Sobelov space containment}, it follows that 
\begin{equation}
\Phi_{\Omega}^{s} : W^{s + N(s)}(\Omega) \rightarrow W^{s}_{0}(\Omega) \label{Bell operator continuous}
\end{equation}
is continuous, where $N(s) = k\nu(s)+(k-1)s$. If $P_{\Omega}$ is the Bergman projection for $\Omega$, Lemma \ref{Condition R product domains} shows that
\begin{align*}
P_{\Omega}\Phi^s &= (P_1 \widehat{\otimes} \ldots \widehat{\otimes} P_m)(\Phi_1^s \widehat{\otimes}\ldots\widehat{\otimes}\Phi_m^s) \\
 &= P_1 \widehat\Phi_1^s \widehat{\otimes}\ldots\widehat{\otimes} P_m \Phi_m^s \\
 &= (P_1 \widehat{\otimes}\ldots\widehat{\otimes} P_m) \\
 &= P_{\Omega}.
\end{align*}
For $f \in A^{\infty}(\Omega)$ and every $s\geq 0$, we have 
\begin{equation}
f = P_{\Omega}f =  P_{\Omega}\Phi^s_{\Omega} f \label{bergman projection and bell operator on H infinity of omega}
\end{equation}
where $\Phi^s_{\Omega}f	 \in \widetilde{W}^{s}_{0}(\Omega) \subset W^{s}_{0}(\Omega)$. Hence $A^{\infty}(\Omega) \subset P_{\Omega}(W^s_0(\Omega))$ for every  $s\geq 0$.

\medskip

It is known that the Bergman kernel $K_{\Omega}(z, a) = P_{\Omega}\psi_a(z)$ where $\psi_a$ is a radially symmetric function with compact support around $a \in \Omega$ such that $\int_{\Omega} \psi_a dV = 1$.  As a byproduct, Lemma \ref{Condition R product domains} shows that $\Omega$ satisfies Condition R and hence $K_{\Omega}(z, a) \in A^{\infty}(\Omega)$ as a function of $z$ for each $a \in \Omega$. Let $\Psi = \{\psi_a : a \in \Omega\}$ and note that $P_{\Omega}\Psi = \{K_{\Omega}(z, a) : a \in \Omega \}$. 

\medskip

To prove the theorem, it suffices to show that the complex linear span of $\Psi$ is dense in $W^{s}_{0}(\Omega)$ for each $s$. Assume this for now and pick $f \in A^{\infty}(\Omega)$. Let $s\geq 0 $ be fixed. In \eqref{Bergman projection and Sobelov spaces}, let $M(s) = ks + km(s)$. Since $f \in A^{\infty}(\Omega)$, we may regard $f$ as an element of $W^{M(s)+N(M(s))}(\Omega)$ and note that by \eqref{bergman projection and bell operator on H infinity of omega}, we get 
\[ f = P_{\Omega}f = P_{\Omega}\Phi_{\Omega}^{M(s)}f \]
where $\Phi_{\Omega}^{M(s)}f \in W_{0}^{M(s)}(\Omega)$.

\medskip

Let $g$ belong to the span of $\Psi$. Then 
\begin{align*}
\Vert f - P_{\Omega}g \Vert_{W^{s}(\Omega)} &= \Vert P_{\Omega}(\Phi^{M(s)}_{\Omega}f - g )\Vert_{W^{s}(\Omega)}\\
&\lesssim \Vert \Phi^{M(s)}_{\Omega}f - g \Vert_{W^{M(s)}(\Omega)}
\end{align*}
where the inequality is a consequence of \eqref{Bergman projection and Sobelov spaces}. Thus, if $g$ is close to $\Phi^{M(s)}_{\Omega}f$ (note that both $g$ and $\Phi^{M(s)}_{\Omega}f$ are in $W^{M(s)}_{0}(\Omega)$!) in $W^{M(s)}_{0}(\Omega)$, then $f$ will be close to $P_{\Omega}g$ in $W^{s}(\Omega)$. It remains to observe that $P_{\Omega}g$ is a linear combination of functions of $z$ of the form $K_{\Omega}(z,a)$ where $a \in \Omega$.

\medskip

To complete the proof, we show that the complex linear span of $\Psi$ is dense in $W^{s}_{0}(\Omega)$ for each $s$. For this it suffices to show that each $f \in C^{\infty}_{0}(\Omega)$ is contained in the closure of $\Psi$ since $W^{s}_{0}(\Omega)$ is by definition the closure of $C^{\infty}_{0}(\Omega)$ in $W^{s}(\Omega)$. Let $\eta$ be a smooth radially symmetric function with support in the unit ball such that $\int_{\mathbb{C}^{n}} \eta = 1$. Let $f_{\epsilon} = f\ast\eta_{\epsilon}$ where $\eta_{\epsilon} = \epsilon^{-2n}\eta(z/\epsilon)$. Then $f_{\epsilon}$ converges to $f$ in $W^{s}_{0}(\Omega)$ and by approximating the integral representation 
\[ f_{\epsilon}(z) = \int\limits_{\Omega}f(y)\eta_{\epsilon}(z-y)dV_{y} \]
by finite Riemann sums of the form
\[\sum c_{\nu}f(\tau_{\nu})\eta_{\epsilon}(z-\tau_{\nu})  \]
for certain constants $c_{\nu}$ and points $\tau_{\nu} \in \Omega$, it follows that 
\[ f(z) \approx f_{\epsilon}(z) \approx \sum c_{\nu}f(\tau_{\nu})\eta_{\epsilon}(z-\tau_{\nu}). \]
Further, since the modulus of the continuity of the function $y \rightarrow f(y)\phi_{x}^{(\alpha)}(x-y)$ can be bounded independently of $x$ and  multi-indices $\alpha$ for $|\alpha| \leq s$, it follows that there exists a finite Riemann sum as above which approximates $f$ in the $W^{s}_{0}(\Omega)$-norm. To conclude, note that each such Riemann sum is in $\Psi$.
\end{proof}

\begin{cor}\label{density of Bergman span}
The Bergman span $\mathcal{K}_{\Omega}$ is dense in $A^{\infty}(\Omega)$ in the $C^{\infty}$-topology.
\end{cor}

%%%%%%%%%%%%%%%%%%%%%%%%%%%%%%%%%%%%%%%%%%%%%%%%%%%%%%%%%%%%%%%%%%%%%%%%%%%%%

\section{Proof of Theorem \ref{Main theorem} and Theorem 1.4}

\noindent Let $D \subset \mathbb{C}^{n-1}, \Omega \in \mathbb{C}$ be as in Theorem \ref{Main theorem}. As explained earlier, the problem is to find a single valued holomorphic function $g$ on $D\times\Omega$ such that 
\[ \frac{\partial g}{\partial z_n} = u \]
for a given $u \in \mathcal{K}_{D\times \Omega}$.

\medskip

First suppose that $\Omega$ is simply connected. Pick a base point $a \in \Omega$ and let $\gamma$ be a path in $\Omega$ joining $a$ to an arbitrary point $z_n \in \Omega$. Then
\begin{equation}
g(z',z_n) = \int\limits_{\gamma}u(z',\lambda)\;d\lambda + a \label{Partial differential equation solution}
\end{equation}
is well defined, i.e., the integral is independent of $\gamma$ since $\Omega$ is simply connected. Moreover, $g$ is holomorphic in the $z'$ variable and also with respect to $z_n$. By 
Hartogs' theorem, it follows that $g$ is holomorphic on $D\times\Omega$. The fact that $g$ is  a solution to \eqref{Partial Differential Equation} is evident. Since 
\[
K_{D\times\Omega}\big((z',z_n),(w',w_n)\big) = K_{D}(z',w')\; K_{\Omega}(z_n,w_n)
\]
and $ K_{D}(z',w') \in A^{\infty}(D), K_{\Omega}(z_n,w_n) \in A^{\infty}(\Omega)$ as a function of $z'$ and $z_n$ respectively, it follows that each member of 
$\mathcal{K}_{D\times\Omega} \in A^{\infty}(D\times\Omega)$. In particular $u \in A^{\infty}(D \times\Omega)$. 
It follows that $g$ defined by \eqref{Partial differential equation solution} 
is in $A^{\infty}(D\times\Omega)$.

\medskip

Now, the holomorphic mapping $f$ defined in \eqref{Graph of f a proper mapping}, i.e.,
\[ f(z) = (z_1, z_2, \ldots, z_{n-1}, g(z))\]
satisfies
\begin{align}
f(z) - (z_1, z_2, \ldots, z_{n-1}, z_n) &= \big(0,\ldots,0,\int\limits_{a}^{z_n}u(z',\lambda)\;d\lambda + a-z_n \nonumber \big) \\
&=  \big(0,\ldots,0,\int\limits_{a}^{z_n}\big(u(z',\lambda)-1\big)\;d\lambda\big).  \label{PDE closed form}
\end{align}

If $u \in \mathcal{K}_{D\times\Omega}$ is arbitrarily close to the constant function $h(z) \equiv 1$ in $A^{\infty}(D\times\Omega)$, then \eqref{PDE closed form} 
shows that $f$ is also arbitrarily close to the identity 
map in $A^{\infty}(D\times\Omega)$. By Proposition \ref{quadrature domains and proper mappings}, it 
follows that $f(D\times\Omega)$ is a quadrature domain that is as close to $D\times\Omega$ as desired.

\medskip

\noindent \textit{\textbf{Proof of Proposition \ref{Solution to PDE}:} }Let $u \in A^{\infty}(D\times\Omega)$. We would like to find a function $v\in A^{\infty}(D\times\Omega)$ such that 
\[ \frac{\partial g}{\partial z_n} = v \]
admits a single valued solution $g:D\times\Omega \rightarrow \mathbb{C}$. Let $\gamma_1, \gamma_2,\ldots,\gamma_{m-1}$ be the inner boundary components of $\Omega$. For a fixed $z' \in D$, the periods of $u(z',z_n)$ as a function of $z_n$ around the $\gamma_j$'s are given by 
\[ c_j(z') = \int\limits_{\gamma_j}u(z',z_n)\; dz_n \]
for $1\leq j \leq m-1$. In particular, the periods $c_j(z') \in A^{\infty}(D)$ since $u \in A^{\infty}(D\times\Omega)$. Let us denote the column matrix 
$(c_1(z'),c_2(z'),\ldots,c_{m-1}(z'))^T$ by $C(z')$. By Lemma 2.1 in \cite{Be-density}, there are points $\z_1, \z_2,\ldots, \z_{m-1} \in \Omega$ such that the $(m-1)\times(m-1)$ matrix 
of periods, 
say $M$ whose $(i,j)$-th entry is 
\[ 
\int\limits_{\gamma_i}K_{\Omega}(z_n, \z_j)\;dz_n
\]
is non-singular. Let
\begin{equation}
v = u - \sum_{j =1}^{m-1}a_j(z')K_{\Omega}(z_n, \z_j) \label{changing u for solution of PDE}
\end{equation}
where $a_1(z'),a_2(z'),\ldots,a_{m-1}(z')$ are unknown holomorphic functions to be determined later. Denote by $A(z')$ the column matrix $(a_1(z'),a_2(z'),\ldots,a_{m-1}(z'))^T$. 
In order that $v$ has zero periods, we must have 
\[ 
0 = \int\limits_{\gamma_i}v(z',z_n)\;dz_n  = c_i(z')- \sum_{j =1}^{m-1}a_j(z')\int\limits_{\gamma_i}K_{\Omega}(z_n, \z_j)\;dz_n 
\]
for all $1\leq i \leq m-1$ and each fixed $z' \in D$. In matrix form, these equations are equivalent to
\[
C(z') = M A(z') 
\]
from which we can solve for $A(z')$ since $M$ is non-singular. Note that the components of $A(z') \in A^{\infty}(D)$ and hence $v \in A^{\infty}(D\times\Omega)$. By repeating the arguments given earlier for the case of a simply connected $\Omega$, we see that 
\[ \frac{\partial g}{\partial z_n} = v \]
has a single valued solution given by \eqref{Partial differential equation solution}. Further, if $f(z') = (z_1,z_2,\ldots,z_{n-1},g(z))$, then  \eqref{PDE closed form} shows that 
\begin{equation}
f(z) - (z_1,z_2,\ldots,z_{n-1},z_n) = \big(0,\ldots,0, \int\limits_{a}^{z_n}(v(z',\lambda) - 1)\;d\lambda \big) \label{last labelled equation}
\end{equation}
where $a\in \Omega$ is an arbitrary base point.	 

\medskip

If $D,\Omega$ are as in Theorem \ref{Main theorem},  it follows that $\mathcal{K}_{D\times\Omega}$ is dense in $ A^{\infty}(D\times\Omega)$ in the $C^{\infty}$-topology. Therefore, we may choose $u \in \mathcal{K}_{D\times\Omega}$ that is as close to the constant function $h(z)\equiv 1$ in $ A^{\infty}(D\times\Omega)$ as desired. Thus the matrix $C(z')$ has arbitrarily small norm uniformly for all $z'\in D$. The same holds for the matrix $A(z')$ since $M$ is a constant matrix. That $M$ is a constant matrix has one more consequence, namely that $v\in \mathcal{K}_{D\times\Omega}$. Indeed, by \eqref{changing u for solution of PDE}, we need to focus only on 
\[ 
\sum\limits_{j=1}^{m-1}a_j(z')K_{\Omega}(z_n,\z_j).
\]
The functions $a_j(z')$ are linear combinations of the functions $c_j(z')$. Since $u \in \mathcal{K}_{D\times\Omega}$, it is built up of terms of the form
\[ 
K^{(\alpha)}_{D\times\Omega}(z,w) = \big(K_{D}(z',w') \; K_{\Omega}(z_n,w_n)\big)^{(\alpha)} 
\]
which by the product rule involves only mixed terms of the form	
\[ 
\frac{\partial^{\alpha_1+\ldots+\alpha_{n-1}}}{\partial\overline w_1^{\alpha_1}\ldots\partial\overline w_{n-1}^{\alpha_{n-1}}}K_{D}(z',w')	
\; \frac{\partial^{\alpha_n}}{\partial\overline w_n^{\alpha_n}}K_{\Omega}(z_n,w_n)	 
\] 
if $\alpha = (\alpha_1,\ldots,\alpha_n)$. In particular, no terms that involve only $K_{D}(z',w')$ (along with its derivatives) or $K_{\Omega}(z_n,w_n)$ (along with its derivatives) 
occur. This shows that apart from constants
\begin{align*}
a_j(z')K_{\Omega}(z_n,\z_j) &\sim \sum c_i(z')K_{\Omega}(z_n,\z_j) \\
&\sim \sum\bigg(\int\limits_{\gamma_i} u(z',z_n)\;dz_n\bigg)K_{\Omega}(z_n,\z_j)	 \\
&\sim \sum\bigg(\int\limits_{\gamma_i} K_{D\times\Omega}^{(\alpha)}(z,w)\;dz_n\bigg)K_{\Omega}(z_n,\z_j)	 \\
&\sim \sum\bigg(\frac{\partial^{\alpha_1+\ldots+\alpha_{n-1}}}{\partial\overline w_1^{\alpha_1}\ldots\partial\overline w_{n-1}^{\alpha_{n-1}}}K_{D}(z',w')\bigg)K_{\Omega}(z_n,\z_j)	 
\\
&\sim K_{D\times\Omega}^{(\alpha_1,\ldots,\alpha_{n-1},0)}(z, (w', \z_j)).
\end{align*}
Hence $v \in \mathcal{K}_{D\times\Omega}$ and by keeping track of the various finitely many constants in the sums above, we can ensure that $v$ is as close the constant function $h(z) 
\equiv 1$ in $A^{\infty}(D\times\Omega)$ by choosing $u$ sufficiently close to $h(z) \equiv 1$ in $A^{\infty}(D\times\Omega)$. By Proposition \ref{quadrature domains and proper mappings}, $f(D\times\Omega)$ (where $f$ is as in \eqref{last labelled equation}) is a quadrature domain that is arbitrarily close to $D\times\Omega$.

\medskip

The proofs of both statements in Theorem 1.4 are exactly the same as before even though the domains being considered are not product 
domains. However, there is a plane of symmetry over which the domains sit, the fibres being conformal discs in the plane. This is enough for the previously used arguments to go through. 
For the sake of definiteness, let us focus on $D\subset \mathbb{C}^2$, a smoothly bounded complete Hartogs domain. $D$ satisfies Condition R and admits a Bell operator. Assume that the 
plane of symmetry is $\{z_2 = 0\}$ and fix a base point, say the origin which is contained in all the discs that lie over the plane of symmetry. If $u \in \mathcal{K}_D$, let
\[ 
g(z_1,z_2) = \int\limits_{0}^{z_2}u(z_1,\lambda)\;d\lambda 
\]
which is a well defined holomorphic function on $D$ by Hartogs theorem. If $u$ is close to the constant function $h(z) \equiv 1$ in $A^{\infty}(D)$, the mapping $f$ defined by \eqref{Graph of f a proper mapping} is then a small perturbation of the identity whose range is a quadrature domain.

%%%%%%%%%%%%%%%%%%%%%%%%%%%%%%%%%%%%%%%%%%

\section{Concluding Remarks}

\noindent The definition of a quadrature domain adopted in this note is equivalent to the condition that the constant function $h(z) \equiv 1$ belong to the Bergman span. It was 
suggested in \cite{Be-density} that a more interesting class of domains to study might be the ones for which all holomorphic polynomials in $\mathbb C^n$ belong to the Bergman span. In 
fact, 
for planar domains these two conditions were shown to be equivalent in \cite{Be-density} and this was done by using the fact that the associated Schwarz function $S(z) = \ov z$ on the 
boundary 
of a planar qudarature domain. We do not know whether this equivalence continues to hold in higher dimensions as well. It is however evident that the property that all holomorphic 
polynomials belong to the Bergman span is apriori more restrictive than simply demanding that the constant function $h(z) \equiv 1$ be in the Bergman span. Nevertheless, 
\cite{Be-circular1}, \cite{Be-circular2} 
show that complete circular domains containing the origin do enjoy the property that all holomorphic polynomials in $\mathbb C^n$ be in the associated Bergman span. In fact, we need 
only look at the complex linear span of functions of $z$ of the form $K(z, 0)$ and $K^{(\alpha)}(z, 0)$ and this in fact contains all the polynomials. These domains were christened 
\textit{one-point quadrature domains} in \cite{Be-density} for this reason. Further, it can be shown that if $f : \Om_1 \ra \Om_2$ is a biholomorphic polynomial mapping 
with unit Jacobian and polynomial inverse, where $\Om_2$ is a complete circular domain containing the origin, then $\Om_1$ is a one-point quadrature domain. Using this, it is possible to 
construct several explicit examples. For instance, we may take $\Om_2$ to be the unit ball $\mathbb B^n$ and $f$ a volume preserving H\'{e}non map of the form
\[
f(z, w) = (w, w^2 - z)
\]
which yields 
\[
f^{-1}(\mathbb B^n) = \big\{(z, w) : \vert z \vert^2 + \vert w \vert^2 + \vert w \vert^4 - 2 \Re(w^2 \ov z) < 1 \big\}. 
\]
Other examples can be constructed by replacing the term $w^2$ in the second component by an arbitrary polynomial $p(w)$. By considering volume preserving shift-like automorphisms of 
$\mathbb C^3$ of the form
\[
f(z_1, z_2, z_3) = (z_2, z_3, z_1 + q(z_3))
\]
where $q(z_3)$ is a polynomial, we get 
\[
f^{-1}(\mathbb B^n) = \big\{z \in \mathbb C^3 : \vert z_1 \vert^2 + \vert z_2 \vert^2 + \vert z_3 \vert^3 + \vert q(z_3) \vert^2 + 2 \Re(z_1 \ov{q(z_3)}) < 1 \big\}. 
\]
Similar examples can be constructed in $\mathbb C^k$ for $k \ge 3$.

%%%%%%%%%%%%%%%%%%%%%%%%%%%%%%%%%%%%%%%%%%%%%%%%%%%%%%%%%%%%%%%%%%%%%%%%%%%%%%%

\end{document}